\documentclass[11pt,A4]{article}
\usepackage{amsmath,amsthm,amsfonts,amssymb,amscd, amsxtra,color}
\usepackage[active]{srcltx}
\usepackage{url}
\usepackage{cite}
\usepackage{multirow}
\usepackage[margin=2.5cm,nohead]{geometry}
\newtheorem{theorem}{Theorem}
\newtheorem{lemma}{Lemma}
\newtheorem{corollary}{Corollary}
\newtheorem{proposition}{Proposition}
\newtheorem{remark}{Remark}

\DeclareMathOperator{\inte}{int}

\DeclareMathOperator{\diag}{diag}
\DeclareMathOperator{\argmax}{argmax}
\DeclareMathOperator{\argmin}{argmin}

\newcommand{\pr}{\perp}

\newcommand{\lng}{\langle}
\newcommand{\rng}{\rangle}
\newcommand{\lf}{\left}
\newcommand{\rg}{\right}
\newcommand{\R}{\mathbb R}

\newcommand{\f}{\frac}

\newcommand{\SP}{\mathbb S}

\begin{document}

\title{On the spherical convexity of  quadratic functions 
\thanks{This work was supported by CNPq (Grants  302473/2017-3 and 408151/2016-1) and FAPEG.}}

\author{
 O. P. Ferreira\thanks{IME/UFG, Avenida Esperan\c{c}a, s/n, Campus II,  Goi\^ania, GO - 74690-900, Brazil (E-mails: {\tt orizon@ufg.br}).}
\and
S. Z. N\'emeth \thanks{School of Mathematics, University of Birmingham, Watson Building, Edgbaston, Birmingham - B15 2TT, United Kingdom
(E-mail: {\tt s.nemeth@bham.ac.uk}).}
}

\maketitle

\begin{abstract}
 In this paper we study the spherical convexity of quadratic functions on spherically convex sets. In particular, conditions characterizing the  spherical convexity of quadratic functions on spherical convex sets associated to the positive orthants and  Lorentz  cone are given. 

\end{abstract}


\section{Introduction} \label{sec:Int}

In this paper we study the spherical convexity of quadratic functions on spherical convex sets. This problem arises when one tries to
make certain fixed point theorems, surjectivity theorems, and existence theorems for complementarity problems and variational inequalities more explicit (see \cite{IsacNemeth2003,IsacNemeth2004,IsacNemeth2005,IsacNemeth2006}). Other results on this subject can also be found in \cite{IsacNemeth2008}.  In particular, some existence theorems could be reduced to optimizing a quadratic function on the intersection of the sphere and a cone. 
%
Indeed, consider a closed convex cone $K\subseteq\R^n$ with dual $K^*$. Let $F:\R^n\to\R^n$ be a continuous mapping such that 
$G:\R^n\to\R^n$ defined by $G(x)=\|x\|^2F(x/\|x\|^2)$ and $G(0)=0$  is differentiable at $0$. Denote by $DG(0)$ the Jacobian matrix of $G$ at $0$.
By \cite[Corollary~8.1]{IsacNemeth2006} and \cite[Theorem18]{Nemeth2006}, if 
$\min_{\|u\|=1,u\in K}\langle DG(0)u, u\rangle>0$, then the  nonlinear 
complementarity problem defined by $K\ni x\perp F(x)\in K^*$ has a solution. 
Thus, we need to minimize a quadratic form on the
intersection between a cone and the sphere. These sets are exactly the spherically convex sets;  see \cite{FerreiraIusemNemeth2014}. Therefore,
this leads to minimizing quadratic functions on spherically convex sets.  In fact the optimization problem above reduces to the problem of
calculating the scalar derivative, introduced by S. Z. N\'emeth in  \cite{Nemeth1992,Nemeth1993,Nemeth1997},  along cones; see \cite{Nemeth2006}.
Similar minimizations  of quadratic functions on spherically convex sets  are needed in the other settings; see
\cite{IsacNemeth2003,IsacNemeth2004,IsacNemeth2005}.  Apart from the above, motivation of this study is much wider.    For instance,  the
quadratic constrained optimization problem on the sphere
\begin{equation}\label{eq:gp}
	\min \{\langle Qx, x\rangle~:~x\in C\}, \qquad  C\subseteq\SP^{n}, 
\end{equation}
for  a symmetric matrix $Q$,  is a minimal eigenvalue problem, that is, finding  the spectral norm of the matrix \(-Q\) (see, e.g.,~\cite{Smith1994}). The problem \eqref{eq:gp} also contains the trust region problem that appears in many nonlinear programming  algorithms as a sub-problem,  see \cite{DennisSchnabel1983}.

It is worth to point out that when a quadratic function is spherically convex (see, for example,  \cite{FerreiraIusemNemeth2014}), then the spherical local minimum is equal to the global minimum.   Furthermore,  convex  optimization problems posed on the sphere, have a specific underlining algebraic structure that could be exploited to greatly reduce the cost of obtaining the solutions; see \cite{Smith1994,So2011,Zhang2003,Zhang2012}. Therefore, it is natural to consider the problem of determining the spherically convex quadratic functions on spherically convex sets. The goal of the paper is to present conditions satisfied by quadratic functions which are spherically convex on spherical convex sets. Besides,  we present conditions characterizing the spherical convexity of quadratic functions on spherically convex sets associated to the  Lorentz  cones and  the positive orthant cone. 

The remainder of this paper is organized as follows. In Section~\ref{sec:NotBasRes}, we recall some notations and basic results used throughout  the paper. In Section~\ref{Sec: GenProp} we present some general properties satisfied by quadratic functions which are spherically convex. In Section~\ref{sec: CharQFPosiOth} we present a condition characterizing the spherical convexity of quadratic functions on the spherical convex set defined by the positive orthant cone. In Section~\ref{sec: CharQFPCircCone} we present a condition characterizing  the spherical convexity of quadratic functions on spherical convex sets defined by   Lorentz cone. We conclude this paper by making some final
remarks in Section~\ref{sec;FinalRemarks}.

\section{Notations and basic results} \label{sec:NotBasRes}

In this section we present  the notations and  some auxiliary results used throughout the paper. Let $\R^n$ be the $n$-dimensional Euclidean 
space with the canonical inner product $\lng\cdot,\cdot\rng$, norm $\|\cdot\|$. Denote by 
$\R^n_+$ the nonnegative orthant and by $\R^n_{++}$
the positive orthant. The notation $x\pr y$ means that $\lng x,y\rng=0$. Denote by $e^i$ the $i$-th canonical unit vector in $\R^n$. The 
{\it unit sphere} is denoted by 
$$
\SP:=\left\{ x\in \R^{n}~:~\|x\|=1\right\}.
$$
The {\it dual cone} of a cone ${\cal{K}} \subset \R^n$ is the cone 
${\cal{K}}^*\!\!:=\!\{ x\in \R^n : \langle x, y \rangle\!\geq\! 0, ~ \forall \, y\!\in\! {\cal{K}}\}.$
Any pointed closed convex cone with nonempty interior will be called {\it proper cone}. ${\cal K}$ is called {\it subdual} if ${\cal K}\subset
{\cal K}^*$, {\it superdual} if ${\cal K}^*\subset {\cal K}$ and {\it self-dual} if ${\cal K}^*={\cal K}$. ${\cal K}$ is called {\it strongly
superdual} if ${\cal K}^*\subset\inte({\cal K})$. The set of all $m \times n$ matrices with real entries is denoted by $\R^{m \times n}$ and
$\R^n\equiv \R^{n \times 1}$. In Section 5 we will also use the identification $\R^n\equiv\R^{n-1}\times\R$, which makes
the notations much easier. The matrix ${\rm I_n}$ denotes the 
$n\times n$ identity matrix.    If $x\in \R^n$ then $\diag (x)$ will denote an
$n\times n$  diagonal matrix with $(i,i)$-th entry equal to $x_i$, for $i=1,\dots,n$.    For  $a\in \R$ and $B\in \R^{(n-1)\times (n-1)}$  we
denote    $\diag(a ,B)\in  \R^{n\times n}$ the matrix defined by 
$$
\diag(a,B):=
\begin{bmatrix}
a & 0 \\
0 & B
\end{bmatrix}.
$$
Recall that a {\it Z-matrix} is a matrix with nonpositive off-diagonal elements. Let ${\cal K} \subset \R^n $ be  a  pointed closed convex cone  with nonempty interior,  the  {\it ${\cal K}$-Z-property} of a matrix $A\in\R^{n\times n}$ means that 
$\lng Ax,y\rng\le0$, for any $(x,y)\in C({\cal K})$, where $C({\cal K}):=\{(x,y)\in\R^n\times\R^n:x\in {\cal K},\textrm{ }y\in {\cal K}^*, ~x\pr
y\}$. The matrix  $A\in\R^{n\times n}$ is said to have the  {\it ${\cal K}$-Lyapunov-like property} if $A$ and $-A$  have the ${\cal
K}$-Z-property, and  is said to be  {\it ${\cal K}$-copositive} if  $\lng Ax,x\rng\geq 0$ for all $x\in {\cal K}$. If ${\cal K}=\R^n_+$, then the
${\cal K}$-Z-property of a matrix coincides with the matrix being a Z-matrix and the ${\cal K}$-Lyapunov-like property with the matrix being
diagonal.

The intersection curve of a plane though the origin of \(\R^{n}\) with the sphere \( \SP\) is called a { \it geodesic}.   A geodesic segment   is said to be {\it minimal} if its arc length is 
equal to the intrinsic  distance between its end points, i.e., if 
\(
\ell(\gamma):=\arccos \langle  \gamma(a), \gamma(b)\rangle, 
\)
where  \(\gamma: [a, b]\to \SP\) is a parametrization of the geodesic segment. {\it Through the paper we will use the same terminology for a geodesic and its parameterization.} The set \(C\subseteq \SP\) is said to be {\it spherically convex} if for any \(x\), \(y\in C\) 
all the minimal geodesic segments joining \(x\) to \(y\)  are contained in \(C\).  Let \( C\subset \SP\) be a spherically convex set and $I\subset \R$ an interval.   The following result is proved in  \cite{FerreiraIusemNemeth2013}.
\begin{proposition} \label{pr:ccs}
Let   $K_C:=\left\{ tp \, :\, p\in C, \; t\in [0, +\infty) \right\}$ be the cone generated by the set $C \subset \SP^n$. The set   $C $ is   spherically convex if and only if  the associated  cone  $K_C$ is convex and pointed.
\end{proposition}

A function \(f:C\to \R\)  is said to be   {\it  spherically convex (respectively, strictly spherically convex) }
if for any minimal geodesic  segment \(\gamma:I\to C\), the composition  \( f\circ \gamma :I\to \R\) is convex (respectively, strictly convex) in
the usual sense.   The next result is an immediate  consequence of  \cite[Propositions~8 and ~9]{FerreiraIusemNemeth2014}.
\begin{proposition} \label{pr:sgmd}
Let ${\cal K}\subset\R^n$ be a proper cone, ${\cal C}=\inte({\cal K})\cap \SP$ and  \(f:{\cal C} \to \R\) a differentiable function. 
 Then, the following statements are equivalent:
	\begin{itemize}
		\item[(i)]   \(f\) is  spherically convex; 
		\item[(ii)] $ \left\langle Df(x)-  Df(y) , x-y\right\rangle +  (\langle x,  y \rangle -1)\left[ \langle Df(x), x \rangle+  \langle Df(y), y \rangle\right] \geq 0$, for all  $x,  y\in {\cal C}$; 
		\item[(iii)]$ \left\langle D^2f(y)x, x \right\rangle - \langle Df(y), y\rangle \geq 0$,   for all $y\in {\cal C}$, $ x\in \SP$ with $x\pr y$.
	\end{itemize}	
\end{proposition} 
It is well known that  if  $Q\in \R^{n\times n}$ is  an orthogonal matrix, then $Q$ defines a linear orthogonal mapping,  which is an isometry of
the sphere.  In the following remark we state some  important properties of the isometries of the sphere, for that,  given  $  {\cal C}\subset
\SP$ and  $Q\in \R^{n\times n}$, we define
$$
 Q { {\cal C}}:=\{ Qx ~:~ x\in {\cal C}\}.
$$
\begin{remark}  \label{re:isom}
Let $Q\in \R^{n\times n}$  be  an orthogonal matrix, i.e., $Q^T=Q^{-1}$, ${\cal C}_1$ and ${\cal C}_2$ be spherically convex sets. Then  ${\tilde
{\cal C}}_2:= Q { {\cal C}_2}$   is a spherically convex set.  Hence, if     ${\tilde {\cal C}}_2\subset \tilde{\cal C}_1$ and   $ f:{\tilde {\cal
C}}_1\to\R$ is a spherically convex function, then   $h:= f\circ Q :{\cal C}_2\to\R$ is also a   spherically convex function.  In
particular,  if ${\tilde {\cal C}}_2= \tilde{\cal C}_1$ then, $f: \tilde{\cal C}_1\to\R$  is spherically convex if,  only if,   $h:= f\circ Q : {\cal C}_{2} \to\R$ is spherically convex.
 \end{remark}
 We will show next a useful property of proper cones which will be used in the Section~\ref{sec: CharQFPCircCone}.
 \begin{lemma}\label{p-l}
	Let ${\cal K}\subset\R^n$ be a proper cone. If $x\in \SP$ and $y\in {\cal K}\cap \SP$ such that $x\pr y$, then $x\notin
	\inte({\cal K}^*)\cup-\inte({\cal K}^*)$. 
\end{lemma}
\begin{proof}
	If $x\in\inte({\cal K}^*)$, then $\lng x,y\rng>0$ and if $x\in-\inte({\cal K}^*)$, then $\lng x,y\rng<0$. Hence, $x\in \SP$, $y\in {\cal K}\cap \SP$ and $x\pr y$
	imply $x\notin  \inte({\cal K}^*) \cup- \inte({\cal K}^*) $.
\end{proof}
Let $\mathcal C\subseteq\mathcal D\subseteq\R^n$ and $A\in\R^{n\times n}$. For a quadratic function $f:\mathcal C\to\R$ defined by $f(x)=\lng Ax,x\rng$, we 
will simply use the notation $f$ for the function $\tilde f:\mathcal D\to\R$ defined by $\tilde f(x)=\lng Ax,x\rng$.

\section{Quadratic functions  on spherical convex sets} \label{Sec: GenProp}
 In this section we present some general  properties satisfied  by  quadratic functions which are spherically  convex.
\begin{proposition}\label{pr:CharMon}
	Let ${\cal K}\subset\R^n$ be a proper cone, ${\cal C}=\inte({\cal K})\cap \SP$ and let $f:{\cal C}\to\R$ be defined by 
	$f(x)=\lng Ax,x\rng$, where $A\in\R^{n\times n}$. Then, the following statements are equivalent:
	\begin{itemize}
		\item[(i)] The function $f$ is spherically convex;
		\item[(ii)]  $ \lng Ax,x\rng-  \lng Ay,y\rng \geq 0$, for all  $x\in \SP$ and  $y\in {\cal K}\cap \SP$  with $x\perp y$.
	\end{itemize}
\end{proposition}
\begin{proof}
To  prove the equivalence of items  (i) and (ii), note that  ${\cal C}=\inte({\cal K})\cap \SP$ is an open spherically convex set,
$Df(x)=2Ax$ and $D^2f(x)=2A$,  for all $x\in {\cal C}$. Then,  from  item (iii) of
Proposition~\ref{pr:sgmd} we conclude that  $\lng Ax,x\rng\geq\lng Ay,y\rng$, for all  $x\in \SP$ and  $y\in {\cal C}$  with $x\perp y$.  Hence,
by  continuity  this inequality  extends for all    $y\in {\cal K}\cap \SP$  with $x\perp y$.  
\end{proof}
\begin{proposition}\label{pr:Char2}
	Let  ${\cal K}\subset\R^n$ be a proper cone, ${\cal C}=\inte({\cal K})\cap \SP$ and let $f:{\cal C}\to\R$ be defined by 
	$f(x)=\lng Ax,x\rng$, where $A=A^T\in\R^{n\times n}$. 
	The following statements are equivalent:
	\begin{itemize}
		\item[(i)] The function $f$ is spherically convex;
		\item[(ii)]   $ 2\left\lng Ax,y\right \rng \leq \left(  \lng Ax,x \rng +  \lng Ay,y \rng   \right)  \lng x,y \rng$, for all $x,y\in  {\cal K}\cap \SP$.
	\end{itemize}
	As a consequence, if ${\cal K}$ is superdual and $f$ is spherically convex, then $A$ has the ${\cal K}$-Z-property.
\end{proposition}
\begin{proof}
First note that, by  taking $f(x)=\lng Ax,x\rng$  the inequality in item (ii) of Proposition~\ref{pr:sgmd}  becomes 
$\left\langle Ax-  Ay , x-y\right\rangle +  (\langle x,  y \rangle -1)\left[ \langle Ax, x \rangle+  \langle Ay, y \rangle\right] \geq 0$, for all  $x,  y\in {\cal C}$. Considering that $A=A^T$, some algebraic manipulations   show that $ 2\left\lng Ax,y\right \rng \leq \left(  \lng Ax,x \rng +  \lng Ay,y \rng   \right)$, for all  $x,  y\in {\cal C}$, and  by  continuity  this inequality  extends for all $x,y\in  {\cal K}\cap \SP$. Terefore, the equivalence of items  (i) and (ii)  follows from  item (ii) of Proposition~\ref{pr:sgmd}. For the second part, let $x\in {\cal K}\cap \SP$ and  $y\in  {\cal K}^* \cap \SP\subset {\cal K}\cap \SP$ with $x\pr y$. Since $f$ is
 spherically convex and $x\pr y$, the inequality in item (ii) implies  $\left\lng  Ax,y\right \rng \leq 0$. Therefore, the result follows
 from the definition of ${\cal K}$-Z-property.
\end{proof}
\begin{proposition}\label{r-l}
	Let ${\cal K}\subset\R^n$ be a superdual proper cone, ${\cal C}=\inte({\cal K})\cap \SP$ and $f:{\cal C}\to\R$ be defined by  $f(x)=\lng
	Ax,x\rng$, where $A=A^T\in\R^{n\times n}$.  If $f$ is spherically convex, then  the following statements hold:
	\begin{itemize}
		\item[(i)] If $x,y\in  ({\cal K}\cup -{\cal K}) \cap \SP $ are such that $x\pr y$, then $\lng Ax,x\rng=\lng Ay,y\rng$; 
		\item[(ii)] If $x\in    \inte({\cal K})   \cap \SP $ and $y\in   {\cal K}\cap \SP$ are  such that $x\pr y$, then $Ax\pr y$;  
		\item[(iii)] If $x\in  -\inte({\cal K})   \cap  \SP$ and $y\in   {\cal K}\cap \SP$ are such that $x\pr y$, then $Ax\pr y$.  
	\end{itemize}
\end{proposition}
\begin{proof}
For proving item  (i), we use  the equivalence of items (i) and (ii) of Proposition \ref{pr:CharMon}  to obtain that  $\lng Ax,x\rng\geq\lng
Ay,y\rng$ and $\lng Ay,y\rng\geq\lng Ax,x\rng$, for all  $x,y\in  ({\cal K}\cup -{\cal K}) \cap \SP $, and the results follows. To prove item
(ii), given $x\in    \inte({\cal K})   \cap \SP $ and $y\in   {\cal K}\cap \SP$ such that $x\pr y$, define $u=(1/(m^2+1))(mx-y)$ and
$v=(1/(m^2+1))(x+my)$, where $m$ is a positive integer. Since $x\in    \inte({\cal K})   \cap \SP $, if $m$ is large enough, then $(1/m)u\in {\cal
K}$ and therefore $u\in {\cal K}$ too.  It is easy to check that $u,v\in   {\cal K}\cap \SP$ such that $u\pr v$. By using item (i)  twice, we
conclude that  $\lng mAx-Ay,mx-y\rng=\lng Ax+mAy,x+my\rng,$ which after some algebraic transformations, bearing in mind that $A=A^T$, implies
  $Ax\pr y$. We can prove item (iii) in a similar fashion.
  \end{proof}
\begin{corollary}\label{cr:stg}
	Let ${\cal K}\subset\R^n$ be a strongly superdual proper cone, ${\cal C}=\inte({\cal K})\cap \SP$ and let $f:{\cal C}\to\R$ be defined by 
	$f(x)=\lng Ax,x\rng$, where $A=A^T\in\R^{n\times n}$.  If $f$ is spherically convex, then $A$ is ${\cal K}$-Lyapunov-like.
\end{corollary}
\begin{proof}
	Let $x\in {\cal K}\cap \SP$ and $y\in {\cal K}^* \cap \SP\subset\inte({\cal K})\cap \SP$ with $x\pr y$. Then, item (ii) of Proposition
	\ref{r-l} implies $Ax\pr y$  and the result follows from the definition of the ${\cal K}$-Lyapunov-like property.
\end{proof}
\begin{proposition}\label{p-c}
	Let ${\cal K}\subset\R^n$  be a superdual proper cone, ${\cal C}=\inte({\cal K})\cap \SP$   and  $f:{\cal C}\to\R$ be defined by  $f(x)=\lng Ax,x\rng$, where $A=A^T\in\R^{n\times n}$. If $A$ is ${\cal K}$-copositive and $f$ is spherically convex, then $A$ is positive 
	semidefinite.
\end{proposition}
\begin{proof}
Since $A$ is ${\cal K}$-copositive  we have $\lng Ax ,x\rng\ge0$ for all  $x\in ( {\cal K^*}\cup -{\cal K}^*) \cap   \SP  \subset (  {\cal  K}\cup -  {\cal K}) \cap   \SP $.   Assume that  $x\in \SP \setminus({\cal K}^*\cup -{\cal K}^*)$.  We claim that, there exists $y\in  {\cal K}\cap \SP$ such that $y\perp x$. We proceed to prove the claim. Suppose that there is no such $y$.  Then, we must have that either $\lng u,x\rng<0$ for all $u\in {\cal K} \setminus\{0\}$, or $\lng u,x\rng>0$ for all $u\in {\cal K} \setminus\{0\}$. If there exist $u\in {\cal K} \setminus\{0\}$ with $\lng u,x\rng<0$ and a $v\in {\cal K} \setminus\{0\}$ with $\lng v,x\rng\ge0$, then $\psi(0)<0$ and $\psi(1)\ge0$, where  the continuous function $\psi:\R\to\R$ is defined by $\psi(t)=\lng (1-t)u+tv,x\rng$. Hence, there is an $s\in [0,1]$ such that $\psi(s)=0$. By	the convexity of ${\cal K} \setminus\{0\}$ (${\cal K} \setminus\{0\}$ is spherically convex because ${\cal K} $ is pointed), we conclude  that $(1-s)u+sv\in {\cal K} \setminus\{0\}$. Let $w=(1-s)u+sv$ and $y=w/\|w\|$. Clearly, $y\in {\cal K} \cap \SP$ and $y\perp x$, which contradicts our assumptions. If $\lng u,x\rng<0$ for all $u\in {\cal K} \setminus\{0\}$, then $x\in -{\cal K}^*$, which is a contradiction. If $\lng u,x\rng>0$ for all $u\in {\cal K} \setminus\{0\}$, then $x\in {\cal K} ^*$, which is also a contradiction. Thus, the claim holds. Since $f$ is convex, Proposition~\ref{pr:CharMon} implies that $\lng Ax,x\rng\ge \lng Ay,y\rng$. Since $A$ is $K$-copositive, we have $\lng Ay,y\rng\ge0$ and hence $\lng Ax,x\rng\ge0$. Thus, $\lng Ax,x\rng\ge0$ for all  $x\in \SP$. In conclusion, $A$ is positive semidefinite.
\end{proof}
By using arguments similar to the ones used in  the proof of Proposition~\ref{p-c} we can also prove the following result.
\begin{proposition}
	Let ${\cal K}\subset\R^n$  be a  subdual proper cone, ${\cal C}=\inte({\cal K})\cap \SP$   and  $f:{\cal C}\to\R$ be defined by  $f(x)=\lng Ax,x\rng$, where $A=A^T\in\R^{n\times n}$. If $A$ is ${\cal K}^*$-copositive and $f$ is spherically convex, then $A$ is positive 
	semidefinite.
\end{proposition}
\section{Quadratic functions on  spherical positive orthant} \label{sec: CharQFPosiOth}

In  this section  we present a  condition characterizing  the  spherical  convexity of quadratic functions on  the spherical convex set associated to the positive orthant  cone.

\begin{theorem} \label{eq:cpoth}
Let ${\cal C}=\SP \cap \R^n_{++}$ and $f:{\cal C}\to\R$ be defined by $f(x)=\lng Ax,x\rng$, where $A=A^T\in\R^{n\times n}$. Then,  $f$ is spherically convex if and only if   there exists $\lambda \in\R$  such that $A=\lambda I_{n}$. In this case,  $f$ is a constant function.
\end{theorem}
\begin{proof}
Assume that there exists $\lambda \in\R$  such that $A=\lambda I_{n}$. In this case,  $f(x)=\lambda$, for all $x\in {\cal C}$. Since   any constant function is spherically convex this implication is proved. For the  converse statement,  we  suppose that $f$ is spherically convex. From the equivalence of items (i) and (ii) of Proposition~\ref{pr:CharMon} we have 
	\begin{equation}\label{eq-char}
		\lng Ax,x\rng\ge\lng Ay,y\rng,
	\end{equation}
	for any $y\in  \R^n_{+}$ and any $x\pr y$ with $x,y\in \SP$. First take $x=e^i$ and $y=e^j$. Then, \eqref{eq-char} implies that
	$a_{jj}\ge a_{ii}$. Hence, by swapping $i$ and $j$, we conclude  that $a_{ii}=\lambda$ for any $i$, where $\lambda\in\R$ is a constant. Next take 
	$y=(1/\sqrt{2})(e^i+e^j)$ and 
	$x=(1/\sqrt2)(e^i-e^j)$. This leads to $a_{ij}\le0$, for any $i,j$. Hence, $A=B+\lambda I_{n}$, where $B$ is a Z-matrix with zero diagonal. It is 
	easy to see that inequality \eqref{eq-char} is equivalent to
	\begin{equation}\label{eq-char2}
		\lng Bx,x\rng\ge\lng By,y\rng,
	\end{equation}
for any $y\in  \R^n_{+}$ and any $x\pr y$ with $x,y\in \SP$. Let $i,j$ be arbitrary but different and $k$ different from both $i$ and $j$.  Let $y=e^k$ and $x=(1/\sqrt2)(e^i+e^j)$. Then,  \eqref{eq-char2} implies that $a_{ij}=b_{ij}\ge0$. Together with $a_{ij}\le0$ this gives $a_{ij}=b_{ij}=0$. Hence $A=\lambda I_{n}$ and therefore $f(x)=\lambda$, for any $x\in {\cal C}$, and the proof is concluded. 
\end{proof}
\section{Quadratic functions on Lorentz   spherical convex sets} \label{sec: CharQFPCircCone}

 In this section we present a    condition characterizing  the  spherical  convexity of quadratic functions on  spherical convex sets associated to the    Lorentz  cones.  We begin with the following definition: Let ${\cal L}\subset\R^n$  be  the {\it Lorentz cone} defined by 
\begin{equation} \label{eq:LorentzCone}
	{\cal L}:=\lf\{x\in\R^n~:~x_1\geq\sqrt{x_2^2+\dots+x_n^2}\rg\}.
\end{equation}
  \begin{lemma}\label{lem:CharLor}
	Let  ${\cal L}$ be  the Lorentz cone,   $x:=(x_1,{\tilde x})$ and $y:=(y_1,{\tilde y})$ in $\SP$. Then the following statements 
	hold:
	\begin{enumerate}
		\item[(i)] $ y \in   - {\cal L} \cup {\cal L}$  if and only if   $y_1^2\ge1/2$.  Moreover,   $y_1^2\ge1/2$ if and only if  $\|{\tilde y}\|^2\le1/2$; 
		\item[(ii)] $ y \in -\inte({\cal L}) \cup \inte({\cal L})$  if and only if     $y_1^2>1/2$. Moreover,  $y_1^2>1/2$ if and only if  $\|{\tilde y}\|^2<1/2$;
		\item[(iii)] $ x \notin -\inte({\cal L}) \cup \inte({\cal L})$ if and only if     $ x_1^2\le1/2$. Moreover, $ x_1^2\le1/2$ if, and only if,  $\|{\tilde x}\|^2\ge1/2$;
		\item[(iv)]  If $ y\in  - {\cal L} \cup {\cal L} $ and $x \perp y$ then $ x\notin -\inte({\cal L}) \cap \inte({\cal L})$.  Moreover,   $ x\notin -\inte({\cal L}) \cap \inte({\cal L})$ if, and only if $ x_1^2\le1/2$. Furthermore,  $ x_1^2\le1/2$ if and only if  $\|{\tilde x}\|^2\ge1/2$.
	\end{enumerate}
\end{lemma}
\begin{proof}
	Items (i)-(iii) follow easily from the definitions of $\SP$ and ${ \cal L}$. Item (iv) follows from Lemma \ref{p-l} and item (iii).
\end{proof}
 \begin{remark} \label{eq:othcha}
Let $\tilde{Q}\in \R^{(n-1)\times (n-1)}$ be  orthogonal. Then,  $Q= \diag(1, \tilde{Q})$ is also ortogonal and  $ Q {\cal L}={\cal L}$. Hence,  from Remark~\ref{re:isom} we conclude that  $f: {\cal L} \cap \SP\to\R$ is spherically convex if,   and only if,    $g:= f\circ Q ={\cal L}\cap \SP \to\R$ is spherically convex. 
\end{remark}
\begin{theorem}\label{th:ccflc}
Let ${\cal C}=\inte( {\cal L} )\cap \SP$ and  $f:{\cal C}\to\R$ be  defined by $f(x)=\lng Ax,x\rng$, where $A=A^T\in\R^{n\times n}$. Then  $f$ is spherically convex  if and only if  there exist $a,\lambda \in\R$ with $\lambda \geq a$ such that $A=\diag(a, \lambda I_{n-1})$.
\end{theorem}
\begin{proof}
Assume that $f$ is spherically convex. Let  $ x, y \in {\cal L}\cap \SP$  with   $x\perp y$  be  defined by
$$
x=\frac{1}{\sqrt{2}}e^{1}+ \frac{1}{\sqrt{2}}e^{i},  \qquad y=\frac{1}{\sqrt{2}}e^{1}- \frac{1}{\sqrt{2}}e^{i},  \qquad i\in \{ 2, \ldots, n\}.
$$
Hence  the item (i) of Proposition \ref{r-l} implies that $\lng Ax,x\rng=\lng Ay,y\rng$. Hence, after computing these inner products, we obtain 
$$
\frac{1}{2}( a_{11}+a_{1i})+ \frac{1}{2}( a_{i1}+a_{ii})= \frac{1}{2}( a_{11}-a_{1i})- \frac{1}{2}( a_{i1}-a_{ii}),  \qquad i\in \{ 2, \ldots, n\}.
$$
Since $A$ is a symmetric matrix, the last equality implies that  $a_{1i}=0$, for all  $i\in \{ 2, \ldots, n\}$. Thus, by letting $a= a_{11}$, we
have $A=\diag(a, \tilde{A})$ with $\tilde{A}\in \R^{(n-1)\times (n-1)}$ a symmetric matrix.  Let $\tilde{Q}\in \R^{(n-1)\times (n-1)}$ be an
orthogonal matrix such that $ \tilde{Q}^T \tilde{A} \tilde{Q}=\Lambda$, where $\Lambda =\diag(\lambda_2, \ldots, \lambda_n)$ and $\lambda_i$  is
an eigenvalue   of  $\tilde{A}$,  for all $i\in \{ 2, \ldots, n\}$. Thus, Remark~\ref{eq:othcha} implies  that   $f:\mathcal L\cap\SP\to\R$ is spherically convex 
if, and  only if,   $g(x)= \langle \diag(a_{11}, \Lambda)x , x \rangle $ is spherically convex.    On the other hand,  using
Proposition~\ref{pr:CharMon}   we conclude that  $g(x)= \langle \diag(a_{11}, \Lambda)x , x \rangle $ is spherically convex if and only if  
$$
h(x)= \lng [\diag(a_{11}, \Lambda)-a_{11}I_{n}]x,x\rng= \lng [\Lambda -a_{11}{I_{n-1}}] \tilde{x}, \tilde{x}\rng,
$$
where $\textrm{ }x:=(x_1, \tilde{x})\in \R\times \R^{n-1}$, is spherically convex. Since $h$  is spherically convex, from 
Proposition~\ref{pr:CharMon} we have 
\begin{equation} \label{eq:ccthflor}
h(x)-h(y)=\lng  [\Lambda -a_{11}{I_{n-1}}] \tilde{x}, \tilde{x}\rng-\lng   [\Lambda -a_{11}{I_{n-1}}] \tilde{y}, \tilde{y} \rng \geq 0,  
\end{equation}
for all points $x=(x_1, \tilde{x})\in \SP$, $y=(y_1, \tilde{y})\in {\cal L}\cap \SP$  with $x\perp y$.  If we assume that $\lambda_2= \ldots =\lambda_n$, we have  $\Lambda=\lambda {I_{n-1}}$ and then $A=\diag(a, \lambda I_{n-1})$, where $a:=a_{11}$ and
$\lambda:=\lambda_2=\dots=\lambda_n$.  Thus  \eqref{eq:ccthflor} becomes $[ \lambda -a_{11}] [
\|\tilde{x}\|^2 -  \|\tilde{y}\|^2] \geq 0$.  Bearing in mind that  ${\cal L}={\cal L}^*$, Lemma~\ref{lem:CharLor} implies  $\|\tilde{x}\|^2 -
\|\tilde{y}\|^2 \geq  0$, and then   we  have from the previous two  inequalities  that  $ a=a_{11}\leq \lambda$. Therefore, for concluding the proof of this  implication  it remains  to prove that  $ a_{11}\leq \lambda_2= \ldots =\lambda_n$.  Without loss of generality we can assume that $n\geq 3$.  Let  $ x\in \SP$ and  $y\in {\cal
L}\cap \SP$  with   $x\perp y$  be  defined by
\begin{align} \label{eq:lcxy1}
x&=-\lf(\f1{\sqrt2}\cos\theta\rg)e^1+\lf(\f12\cos\theta-\f1{\sqrt2}\sin\theta\rg)e^i +\lf(\f12\cos\theta+\f1{\sqrt2}\sin\theta\rg)e^j, \\
 y&=\f1{\sqrt2}e^1+\f12 e^i+\f12 e^j,  \label{eq:lcxy2}
\end{align}
 where $\theta\in (0,\pi)$. From \eqref{eq:lcxy1} and \eqref{eq:lcxy2}, it is straightforward to check that $x\in \SP$, $y\in {\cal L}\cap \SP$
 and  $x\perp y$. Hence, \eqref{eq:ccthflor} becomes  
	\[
	\lf(\f14\sin^2\theta-\f1{\sqrt2}\cos\theta\sin\theta\rg)\lambda_i+\lf(\f14\sin^2\theta+\f1{\sqrt2}\cos\theta\sin\theta\rg)\lambda_j
	\geq0,
	\]
	or, after dividing by $\sin\theta\ne0$, that
	\begin{equation*}
		\lf(\f14\sin\theta-\f1{\sqrt2}\cos\theta\rg)\lambda_i+\lf(\f14\sin\theta+\f1{\sqrt2}\cos\theta\rg)\lambda_j\geq0.
	\end{equation*}
Letting $\theta$ goes to $0$ in the inequality above, we obtain  $\lambda_j\geq\lambda_i$. Hence,   by swapping $i$ and $j$ in
\eqref{eq:lcxy1} and \eqref{eq:lcxy2} we  can also prove that  $\lambda_{i} \geq \lambda_{j}$,  and then $\lambda_{i} = \lambda_{j}$,   for all
$i, j\neq 1$.  Therefore,  $\lambda_2= \ldots =\lambda_n$ which concludes  the implication. Conversely, assume that $A=\diag(a, \lambda
I_{n-1})$ and $\lambda \geq a$. Then $f(x)= \lng [\diag(a, \lambda I_{n-1}]x,x\rng$ and  Proposition~\ref{pr:CharMon}  implies that $f$ is
spherically convex if, and  only if,  
$$
h(x)= \lng [\diag(a, \lambda I_{n-1})-aI_{n}]x,x\rng= \lng [\lambda -a]{I_{n-1}}\tilde{x}, \tilde{x}\rng, 
$$
where $ x:=(x_1, \tilde{x})\in \R\times \R^{n-1} $, is spherically convex. Take  $x=(x_1, \tilde{x})\in \SP$ and  $y=(y_1, \tilde{y})\in {\cal L}\cap \SP$  with $x\perp y$. Thus,  from   Lemma~\ref{p-l} and \eqref{eq:LorentzCone} we have $\|\tilde{x}\|^2 \geq   \|\tilde{y}\|^2.$
Hence considering that    $a\leq \lambda$ we  conclude that   
$$
\lng [\lambda -a]{I_{n-1}}\tilde{x}, \tilde{x}\rng-  \lng [\lambda -a]{I_{n-1}}\tilde{y}, \tilde{y}\rng= [ \lambda -a] [ \|\tilde{x}\|^2 -  \|\tilde{y}\|^2] \geq 0.
$$ 
Therefore, Proposition~\ref{pr:CharMon}  implies   that $h$ is spherically convex and then  $f$ is also spherically convex.
\end{proof}

\begin{remark} 
	Assume  that $f$  in Theorem~\ref{th:ccflc}  is spherically convex  in $  {\cal L}\cap \SP$. Hence  there exist $a,\lambda \in\R$ with
	$\lambda \geq a$ such that $A=\diag(a, \lambda I_{n-1})$ and then  $f(x)= ax_{1}^2+ \lambda \|\tilde{x}\|^2=\lambda-(\lambda-a)x_1^2$, where 
	$x:=(x_1,\tilde{x})\in\mathcal L\cap\SP$. Hence, it is clear that the minimum of $f$ on $\mathcal L\cap\SP$ is obtained when $x_1$ is
	maximal, that is, when $x_1=1$, which happens exactly when $x=e^1$. Similarly, the maximum of $f$ on $\mathcal L\cap\SP$ is obtained when $x_1$ is
	minimal, that is, when $x_1=1/\sqrt2$ (see item (i) of Lemma \ref{lem:CharLor}), which happens exactly when $\|\tilde x\|=x_1=1/\sqrt2$.
	Hence, $\argmin \{ f (x) : x \in  {\cal L}\cap \SP\}=e^1$, $\min \{f (x) : x \in  {\cal L}\cap \SP\}=a$,  $\argmax \{ f (x) : x \in  {\cal L}\cap \SP\}=\lf\{\f1{\sqrt2}(1,\tilde x)\in\R\times\R^{n-1}:\|\tilde x\|=1\rg\}$ and $\max \{f (x) : x \in  {\cal L}\cap \SP\}=(a+\lambda)/2$.
\end{remark}

\begin{remark}
	If $\lambda> a$ then Theorem~\ref{th:ccflc} implies that   $f(x)= \langle\diag(a, \lambda I_{n-1})x, x\rangle$ is  spherically convex.
	However,  in this case   $\diag(a,\lambda, \ldots, \lambda)$ does  not have the ${\cal L}$-Lyapunov-like property. Hence, 
	Corollary~\ref{cr:stg} is not true if we only require  that the cone is superdual proper.  Indeed, the Lorentz cone ${\cal L}$ is 
	self-dual proper, i.e., ${\cal L}^*={\cal L}$ and consequently is superdual proper. Moreover,  letting   $ x, y \in {\cal L}\cap \SP$  
	with   $x\perp y$  be  defined by
	$$
		x=\frac{1}{\sqrt{2}}e^{1}+ \frac{1}{\sqrt{2}}e^{i},  \qquad y=\frac{1}{\sqrt{2}}e^{1}- \frac{1}{\sqrt{2}}e^{i},  \qquad i\in \{ 2,
		\ldots, n\},
	$$
we have $\lng \diag(a, \lambda I_{n-1})x,y\rng=(a-\lambda)/2<0$. Therefore, $\diag(a, \lambda I_{n-1})$ does not have the ${\cal L}$-Lyapunov-like
property, and the strong superduality of the cone   is necessary in  Corollary~\ref{cr:stg}. 
\end{remark}

\section{ Final remarks} \label{sec;FinalRemarks}
This paper is a continuation of   \cite{FerreiraIusemNemeth2013,FerreiraIusemNemeth2014}, where we studied some basic intrinsic properties of spherically convex functions on spherically convex sets of the sphere.  We expect that the results of this paper can  aid in the understanding of the behaviour of spherically convex functions on spherically convex sets of the sphere. In the future we will also study  spherically quasiconvex functions \cite{Nemeth1998} (see also  \cite{KaramardianSchaible1990} for the definition of quasiconvex functions) on spherically convex sets of the sphere.

\section*{Acknowledgments} The authors are grateful to Michal Ko\v cvara and Kay Magaard for many helpful conversations.

\end{document}